\documentclass[12pt,reqno]{amsart}
\usepackage{amssymb}

\newtheorem{Pa}{Paper}[section]
\newtheorem{theorem}[Pa]{{\bf Theorem}}
\newtheorem{lemma}[Pa]{{\bf Lemma}}
\newtheorem{corollary}[Pa]{{\bf Corollary}}
\newtheorem{remark}[Pa]{{\bf Remark}}
\newtheorem{proposition}[Pa]{{\bf Proposition}}

\newtheorem{definition}[Pa]{{\bf Definition}}

\newcommand{\PR}{{\bf NP}}

\newcommand{\cE}{{\mathcal E}}

\newcommand{\B}{{\mathbb B}}
\newcommand{\rH}{{\rm H}^2}

\begin{document}

\author[D. Alpay]{Daniel Alpay}
\address{(DA) Department of Mathematics\\
Ben-Gurion University of the Negev\\
Beer-Sheva 84105 Israel} \email{dany@math.bgu.ac.il}

\author[V. Bolotnikov]{Vladimir Bolotnikov}
\address{(VB)Department of Mathematics \\
The College of William and Mary\\
Williamsburg, VA 23187-8795\\
USA} \email{vladi@math.wm.edu}
\author[F. Colombo]{Fabrizio Colombo}
\address{(FC) Politecnico di
Milano\\Dipartimento di Matematica\\Via E. Bonardi, 9\\20133
Milano, Italy}
\email{fabrizio.colombo@polimi.it}
\author[I. Sabadini]{Irene Sabadini}
\address{(IS) Politecnico di
Milano\\Dipartimento di Matematica\\Via E. Bonardi, 9\\20133
Milano, Italy}
\email{irene.sabadini@polimi.it}

\title[Nevanlinna--Pick interpolation problem]{Self-mappings of the
quaternionic unit ball: multiplier properties, Schwarz-Pick inequality,  and
Nevanlinna--Pick interpolation problem}

\begin{abstract}
We study several aspects concerning slice regular functions
mapping the quaternionic open unit ball $\mathbb B$ into itself. We characterize these functions in
terms of their Taylor coefficients  at the origin and identify them as
contractive multipliers of the Hardy space ${\rm H}^2(\mathbb B)$. In addition, we formulate and solve
the Nevanlinna-Pick interpolation problem in the class of such functions presenting necessary and sufficient
conditions for the existence and for the uniqueness of a solution. Finally, we describe all
solutions to the problem in the indeterminate case.
\end{abstract}

\maketitle
\section{Introduction}
\setcounter{equation}{0}

Let $\mathbb H$ be the algebra of real quaternions
$p=x_0+ix_1+jx_2+kx_3$  where $x_\ell\in\mathbb R$ and $i$, $j$,
$k$ are imaginary units such that $ij=k$, $ki=j$, $jk=i$ and
$i^2=j^2=k^2=-1$.  The conjugate, the absolute value, the real
part and the imaginary part of a quaternion $p$   are defined as
$\bar p=x_0-ix_1-jx_2-kx_3$,
$|p|=\sqrt{x_0^2+x_1^2+x_2^2+x_3^2}$, ${\rm Re}\, p=x_0$ and ${\rm
Im}\, p=ix_1+jx_2+kx_3$, respectively By $\mathbb S$ we denote the
unit sphere of purely imaginary quaternions. Any $I\in\mathbb S$
is such that $I^2=-1$ so that the set ${\mathbb C}_I=\{x+Iy: \,
x,y\in{\mathbb R}\}$ can be identified with the complex plane. We
say that two quaternions $p$ and $q$ are {\em equivalent} if
$p=h^{-1}qh$ for some nonzero $h\in\mathbb H$. Two quaternions
$p$ and $q$ are equivalent if and only if ${\rm Re} \, p ={\rm
Re}\, q$ and $|{\rm Im}\, p|=|{\rm Im}\, q|$ so the set of all
quaternions equivalent to a given $p\in\mathbb H$ form a
$2$-sphere which will be denoted by $[p]$.

\smallskip

Since the algebra $\mathbb H$ is not commutative, function theory
over $\mathbb H$ is quite different from that over the complex
field. There are several notions of regularity for $\mathbb H$-valued functions. The most
notable are due to Moisil \cite{moisil}, Fueter \cite{fuet1, fuet2}, and Brackx, Delanghe, Sommen \cite{bds}.
More recently, upon refining and developing Cullen's approach \cite{cullen}, Gentili and Struppa introduced
in \cite{genstr} the notion of slice regularity which comprises quaternionic polynomials and power
series with quaternionic coefficients on one side. We recall it now.	
\begin{definition}
{\rm Given an open set $\Omega\subset\mathbb H$, a real differentiable function
$f: \, \Omega\to\mathbb H$ is called {\em left slice
regular} (or just {\em slice regular}, in what follows) on $\Omega$ if for every
$I\in\mathbb S$,
\begin{equation}
\frac{\partial}{\partial x}f_I(x+Iy)+I\frac{\partial}{\partial y}f_I(x+Iy)\equiv 0,
\label{1.1}
\end{equation}
where $f_I$ stands for the restriction of $f$ to $\Omega\cap\mathbb C_I$}.
\label{D:lr}
\end{definition}
We will denote by ${\mathcal R}(\Omega, \widetilde\Omega)$ the set of all
functions $f: \, \Omega\mapsto \widetilde\Omega\subset \mathbb H$ which are
(left) slice regular on $\Omega$ and we will write $\mathcal R(\Omega)$ in case
$\widetilde\Omega=\mathbb H$. It is clear that $\mathcal R(\Omega)$
is a right quaternionic vector
space. As was shown in \cite{genstr}, {\em the identity \eqref{3.1}
holds for a fixed $I\in\mathbb S$ if and only if for all
$J\in\mathbb S$ orthogonal to $I$, there exist complex-valued
holomorphic functions $F,G: \, \Omega\cap\mathbb C_I\to\mathbb C_I$
such that $f_I(z)=F(z)+G(z)J\; $ for all $\; z=x+yI\in
\Omega\cap\mathbb C_I$}.

\smallskip

The latter result called "the Splitting Lemma" clarifies the
relation between the restriction of slice regular functions to a complex plane and complex
holomorphy. It allows to get some of the analogs of basic principles of classical complex
analysis (e.g. the uniqueness theorem, the maximum-minimum modulus principle),
in the quaternionic setting. The theory of slice regular
functions is a very active and fast developing area of analysis;
we refer to recent books \cite{css, gss} and references therein.

\smallskip

The parallels with the classical complex analysis become even stronger if one
focuses on functions defined and slice regular on the unit ball $\mathbb B=\{p\in\mathbb H: \, |p|<1\}$.
 Similarly to the complex case, the functions $f\in \mathcal R(\mathbb B)$
admit power series expansion
\begin{equation}
f(p)=\sum_{k=0}^\infty p^kf_k\quad (f_k\in\mathbb H)
\label{1.3}
\end{equation}
where the series on the right converges to $f$ uniformly on compact subsets of $\mathbb B$; on the
other hand, if
$\limsup_k|f_k|^{\frac{1}{k}}\le 1$, the power series as in \eqref{1.3} converges absolutely on
$\mathbb B$ and represents a
slice regular function. We thus may identify the function from $\mathcal R(\mathbb B)$ with power
series of the form \eqref{1.3}
with radius of convergence at least one.

\smallskip

The prominent role played in the classical complex analysis by
analytic self-mappings of the open unit disk is well known. It is thus not surprising that the class
$\mathcal R(\mathbb B, \mathbb B)$ of slice regular self-mappings of the quaternionic unit ball $\mathbb B$
have already attracted much attention. A number of results in this direction
(e.g., M\"obius transformations, Schwarz Lemma, Bohr's inequality)
are presented in \cite[Chapter 9]{gss}. Among other results, we mention realizations for slice regular
functions \cite{acs1}, Schwarz-Pick Lemma \cite{bist}  and Blaschke products \cite{acs2}.

\smallskip

The present paper initiates the systematic study of interpolation theory in the class $\mathcal R(\mathbb B,
\mathbb B)$. In this context, it is convenient to extend the class
$\mathcal R(\mathbb B, \mathbb  B)$ by unimodular constant functions. By the maximum modulus principle,
this extended class
equals $\mathcal R(\mathbb B, \overline{\mathbb B})$, i.e., it consists of functions $f\in\mathcal R(\mathbb
B)$ such that $|f(p)|\le 1$ for all $p\in\mathbb B$. Although the unimodular constant case can be easily
singled out, the results for the extended class $\mathcal R(\mathbb B, \overline{\mathbb B})$ look more
uniform as we will see below.  Our first result
(the analog of the celebrated result of I. Schur \cite{Schur})
characterizes functions from  $\mathcal R(\mathbb B, \overline{\mathbb B})$ in terms of their
Taylor coefficients at the origin.
\begin{theorem}
Let $S$ be slice regular on $\mathbb B$ and let  ${\bf S}_n$ be the lower triangular
Toeplitz matrix given by
\begin{equation}
{\bf S}_n=\left[\begin{array}{cccc}S_{0} & 0 & \ldots & 0
\\ S_{1}& S_{0} & \ddots & \vdots \\ \vdots& \ddots & \ddots & 0
\\ S_{n}&  \ldots & S_{1} & S_{0}\end{array}\right],\quad\mbox{where}\quad S(p)=\sum_{k=0}^\infty
p^kS_k.
\label{1.4}
\end{equation}
The function $S$ belongs to $\mathcal R(\mathbb B, \overline{\mathbb B})$ if and only if the matrix ${\bf
I}_n-{\bf S}_n{\bf S}_n^*$  is positive semidefinite for all integers $n\ge 0$.
\label{T:1.1}
\end{theorem}
Here and in what follows, the symbol ${\bf I}_n$ denotes the $n\times n$ identity matrix.
The notions of adjoint matrices, of  Hermitian matrices, of
positive semidefinite and positive definite matrices over $\mathbb H$ are similar to those over
$\mathbb C$.

\smallskip

The Hardy space ${\rm H}^2(\mathbb B)$ of slice regular square summable power series has been recently
introduced in \cite{acs2}. Our second objective is to identify the class $\mathcal R(\mathbb B,
\overline{\mathbb B})$  with the unit ball of the multiplier algebra of ${\rm H}^2(\mathbb B)$. This in turn
will enable us to apply operator-theoretic tools to solve the quaternionic version of the Nevanlinna-Pick
interpolation problem (we refer to \cite{pick, nevan1} for the classical origins) which is
the main objective of the present paper and which is formulated as follows.

\medskip

$\PR$: {\em Given $n$ distinct points $p_1,\ldots,p_n\in \mathbb B$ and given $n$ target values
$s_1,\ldots, s_n\in{\mathbb H}$, find a function  $S\in\mathcal R(\mathbb B, \overline{\mathbb B})$ such that
\begin{equation}
S(p_i)=s_{i}\quad\quad{for}\quad i=1,\ldots,n.
\label{1.5}
\end{equation}}
The interpolation problem is called {\em determinate} if it has a
unique solution. By the convexity of the solution set, the
indeterminate problem always has infinitely many solutions.
The standard questions arising in any interpolation context are:
\begin{enumerate}
\item Find necessary and sufficient conditions for the problem to have a solution (the solvability
criterion).
\item Find necessary and sufficient conditions for the problem to have a unique solution (the
determinacy criterion).
\item Describe all solutions in the indeterminate case.
\end{enumerate}
As in the classical case the answers for these questions can be given in terms of the
{\em Pick matrix} $P$ of the problem which we define from the interpolation data as follows:
\begin{equation}
P=\left[\sum_{k=0}^\infty p_i^k(1-s_i\bar s_j)\bar p_j^k\right]_{i,j=1}^n.
\label{1.6}
\end{equation}
We remark that infinite series in \eqref{1.6} converge absolutely since $|p_i|<1$ for all
$i=1,\ldots,n$ and
that the diagonal entries of $P$ are equal to
\begin{equation}
P_{ii}=\frac{1-|s_i|^2}{1-|p_i|^2}\quad\mbox{for}\quad i=1,\ldots,n.
\label{1.7}
\end{equation}
Our next result is the following analog of the classical Nevanlinna-Pick theorem.
\begin{theorem}
The problem $\PR$ has a solution if and only if the Pick matrix $P$ is positive semidefinite.
\label{T:1.2}
\end{theorem}
\begin{remark}
{\rm It is known that the restriction of any slice regular function $S$ to any $2$-sphere is
completely determined by the
values of $S$ at any two points of this sphere. Thus, if three interpolation nodes (say,
$p_1,p_2,p_3$) are equivalent,
then the value of $s_3$ must be uniquely specified by $s_1$ and $s_2$  in order the problem to have
a solution.
Theorem \ref{T:1.2} asserts that condition $P\ge 0$ specifies $s_3$ in this unique way. }
\label{R:1.3}
\end{remark}
Therefore, once we know that the problem $\PR$ is solvable, there is no need to keep more than two
interpolation conditions on the
same $2$-sphere. For each set of more than two conditions on the same $2$-sphere, we keep any two of
them and remove the others.
In this way we reduce the original problem to the one for which
$$
{\bf (A)} :\quad \mbox {{\em None three of the interpolation nodes belong to the same $2$-sphere}}.
$$
By Remark \ref{R:1.3} the reduced problem will have the same solution set as the original one. Thus
it is sufficient to get the
uniqueness criterion and the description of the solution set for the reduced problem which is
characterized by the property ${\bf (A)}$.
\begin{theorem}
Under assumption  ${\bf (A)}$, the problem $\PR$ is
determinate if and only if the Pick matrix $P$ of the problem
is positive semidefinite and singular.
\label{T:1.4}
\end{theorem}
A fairly explicit formula for this unique solution will be given in Lemma \ref{L:4.7} below.

\smallskip

The outline of the paper is as follows. In Section 2 we recall some needed background on
slice regular functions and positive kernels. In Section 3 we characterize functions $f\in\mathcal
R(\mathbb B, \overline{\mathbb B})$
as contractive multipliers of the Hardy space $\rH(\mathbb B)$ of the unit ball and prove Theorem
\ref{T:1.1}. In Section 4 we characterize solutions to the problem $\PR$ in terms of positive kernels of
certain structure. Using this characterization,
in Section 5 we give a linear fractional parametrization of all solutions to the problem $\PR$ in
the indeterminate case and recover the Schwarz-Pick lemma as a consequence of this
description. Finally the determinate case of the problem $\PR$ is handled in  Section 6.

\section{Slice hyperholomorphic functions and kernels}
\setcounter{equation}{0}

In this section we collect  a number of basic facts needed in the sequel.
Interpreting the set  $\mathcal R(\mathbb B)$ as the right quaternionic vector space of power series
\eqref{3.4}  converging in $\mathbb B$, one can introduce the ring structure on $\mathcal R(\mathbb
B)$
using the convolution multiplication
\begin{equation}
g\star f (p)=\sum_{k=0}^\infty p^k\cdot \left(\sum_{r=0}^k g_r f_{k-r}\right)\quad\mbox{if}\quad
f(p)=\sum_{k=0}^\infty p^kf_k, \; \;  g(p)=\sum_{k=0}^\infty p^kg_k
\label{2.1}
\end{equation}
which is called  (left) {\em slice regular multiplication} in the present context. As a convolution
multiplication of the power
series over the noncommutative ring, the $\star$-multiplication  is associative and noncommutative.
Since the series in \eqref{2.1} converge absolutely in
$\mathbb B$, we can rearrange the terms getting
\begin{equation}
g\star f (p)=\sum_{k=0}^\infty p^k\left(\sum_{r=0}^\infty p^rg_r\right)f_k=\sum_{k=0}^\infty p^k
g(p)f_k
\label{2.2}
\end{equation}
which can also be written as
\begin{equation}
g\star f (p)=g(p)\sum_{k=0}^\infty (g(p)^{-1}pg(p))^kf_k=g(p)f(g(p)^{-1}pg(p)) \qquad (g(p)\neq 0),
\label{2.3}
\end{equation}
with the understanding that $g\star f (p)=0$ whenever $g(p)=0$. Observe that the point-wise formula
\eqref{2.3} makes sense even if
the functions $f$ and $g$ are not in $\mathcal R(\mathbb B)$
whereas formula \eqref{2.2} makes sense only for $f\in\mathcal R(\mathbb B)$. We also observe that
$g\star f(x)=g(x)f(x)$ for every $x\in\mathbb R$.

\smallskip

If the function $f\in\mathcal R(\mathbb B)$ is as in \eqref{2.1}, then we can construct its slice
regular inverse $f^{-\star}$ as
$f^{-\star}(p)=(f^c\star f)^{-1}f^c(p)$ where the {\em slice regular conjugate} $f^c$ of $f$ is defined by
\begin{equation}
f^c(p)=\sum_{k=0}^\infty p^k\overline{f}_k  \quad\mbox{if}\quad
 f(p)=\sum_{k=0}^\infty p^kf_k
\label{2.5}
\end{equation}
and $f^{-\star}$ is defined in $\mathbb B$ outside the zeros of $f^c\star f$.
If $f$
satisfies $f(0)=f_0\neq 0$,
one can define its $\star$-inverse $f^{-\star}$ using  the power series
$$
f^{-\star}(p)=\sum_{k=0}p^k a_k,\quad\mbox{where}\quad a_0=f_0^{-1}\quad\mbox{and}\quad
a_k=-f_0^{-1}\sum_{j=1}^kf_ja_{k-j}\; \;
(k\ge 1)
$$
with the coefficients $a_k$ defined recursively. If $f(p)\neq 0$ for all $p\in\mathbb B$, the
latter power series converges on $\mathbb B$. Equalities $f^{-\star}\star f=f\star f^{-\star}\equiv
1$ and
$(g\star f)^{-\star}=f^{-\star}\star g^{-\star}$ are immediate.  An application of \eqref{2.3} shows
that
\begin{equation}
f^{-\star}(p)=f(\widetilde{p})^{-1},\quad\mbox{where}\quad \widetilde
p=f^c(p)^{-1}pf^c(p),\quad
 f(p)=\sum_{k=0}^\infty p^kf_k.
\label{2.4}
\end{equation}
\subsection{Right slice regular functions} A real differentiable function $f: \, \Omega\to \mathbb
H$ is called
{\em right slice regular} on $\Omega$ (in notation, $f\in\mathcal R^r(\Omega)$) if for every
$I\in\mathbb S$ its restriction $f_I$ to $\Omega\cap \mathbb C_I$ is subject to
$$
\frac{\partial}{\partial x}f_I(x+Iy)
+\frac{\partial}{\partial y}f_I(x+Iy)I\equiv 0.
$$
The results for right slice regular functions are completely parallel to those for (left) regular
ones.
The functions in $f\in\mathcal R^r(\mathbb B)$ can be identified with power series
$\; f(p)=\sum_{k=0}^\infty f_k p^k\; $
converging on $\mathbb B$. The set $\mathcal{R}^r(\mathbb B)$ itself is a left quaternionic vector
space and it becomes a ring
once we introduce the {\em right slice multiplication}
$$
g\star_r f (p)=\sum_{k=0}^\infty \left(\sum_{r=0}^k g_r f_{k-r}\right)p^k\quad\mbox{if}\quad
f(p)=\sum_{k=0}^\infty f_kp^k, \; \;  g(p)=\sum_{k=0}^\infty g_k p^k
$$
which can be written alternatively (analogously to formulas \eqref{2.2} and \eqref{2.3}) as
\begin{equation}
g\star_r f (p)=\sum_{k=0}^\infty g_kf(p)p^k=\left\{\begin{array}{cll}  g(f(p)pg(p)^{-1}) f(p)
& {\rm if} & f(p)=0, \\ 0 & {\rm if} & f(p)\neq 0.\end{array}\right.
\label{2.6}
\end{equation}

\subsection{Positive kernels} A matrix-valued function $K(p,q): \, \Omega\times\Omega\to \mathbb
H^{m\times m}$ is called a
positive kernel (in notation, $K\succeq 0$) if the block matrix $\left[K(q_i,q_j)\right]_{i,j=1}^r$
is positive semidefinite for any
choice of finitely many points $q_1,\ldots,q_r$. Equivalently,
$$
\sum_{i,j=1}^rc_i^* K (q_i,q_j)c_j\ge 0\quad\mbox{for all}\quad r\in\mathbb N, \;
c_1,\ldots,c_r\in\mathbb H^m, \;
q_1,\ldots,q_r\in\Omega.
$$
\begin{definition}
We say that the kernel $K(p,q): \, \Omega\times\Omega\to \mathbb H^{m\times m}$ is {\em slice
sesquiregular}
on an open set $\Omega\subset\mathbb H$ if it is (left) slice regular in $p$ and right slice regular
in $\bar q$.
\end{definition}

\smallskip

Several simple statements on positive kernels are collected in the next proposition.
\begin{proposition}
Let $\Omega\subset\mathbb H$ and let $K: \, \Omega\times \Omega\to \mathbb H^{m\times m}$ be a
positive kernel.
Then
\begin{enumerate}
\item For every $A: \, \Omega\to {\mathbb H}^{k\times m}$,
 the kernel $A(p)K(p,q)A(q)^*$ is positive on $\Omega\times\Omega$.
\item For every positive definite matrix $P\in {\mathbb H}^{k\times k}$ and any function $B: \,
\Omega\to {\mathbb H}^{m\times k}$, the kernel $\left[\begin{smallmatrix}P & B(q)^*\\
B(p) & K(p,q)\end{smallmatrix}\right]$ is positive if and only if the Schur complement of $P$
defined below is positive semidefinite:
$$
K(p,q)-B(p)P^{-1}B(q)^*\succeq 0.
$$
\item If in addition, $m=1$, $\Omega$ is open and contains the origin,
and $K: \, \Omega\times\Omega\to \mathbb H$ is  slice sesquiregular, then for every
(left) slice regular function $A: \, \Omega\to {\mathbb H}$, the kernel
$(A\star K\star_r \overline{A})(p,q)$ is positive and  slice sesquiregular.
\end{enumerate}
\label{P:2.1}
\end{proposition}
\begin{proof}
Statement (1) follows by the definition of the
positive kernel and the corresponding property of positive
semidefinite matrices. Due to factorization
$$
\begin{bmatrix}P & B(q)^*\\ B(p) & K(p,q)\end{bmatrix}=A(p)
\begin{bmatrix}P & 0\\ 0 & K(p,q)-B(p)P^{-1}B(q)^*\end{bmatrix}
A(q)^*,
$$
where $A(p)=\left[\begin{smallmatrix} \mathbf{I}_k & 0 \\ B(p)P^{-1} &
\mathbf{I}_m\end{smallmatrix}\right]$, part (2) follows from part (1) and the fact that the matrix $A(p)$
is invertible for every $p\in\Omega$. For part (3), see \cite[Proposition 5.3]{acs3}.
\end{proof}

%\begin{remark}{\rm Statement (3) in Proposition \ref{P:2.1} can be proved, with the same technique,
%also for an axially symmetric s-domain $\Omega\subseteq\mathbb H$.}
%\end{remark}
%We remark that the statements in the above proposition are given in the form needed for our current
%purposes.
%In fact, all the statements hold true for operator-valued kernels; it is clear that justification of
%the third statement
%even in the matrix-valued case requires different arguments. In addition the second statement can be
%modified to cover
%the case where $P$ is positive semidefinite. Such  more general results are not needed, however, in
%the present scalar-valued context.

\section{The space $\rH(\mathbb B)$ and its contractive multipliers}
\setcounter{equation}{0}

In this section we show that the class $\mathcal R(\mathbb B, \overline{\mathbb B})$ can be identified with
the class of contractive
multipliers of the
quaternionic Hardy space  $\rH$ of the unit ball $\mathbb B$. This space is defined as the space of
square summable  (left) slice regular power series:
\begin{equation}
\rH=\left\{f(p)=\sum_{k=0}^\infty p^kf_k: \, \|f\|^2_{\rH}:=\sum_{k=0}^\infty|f_k|^2<\infty\right\}.
\label{3.1}
\end{equation}
The space $\rH$ is a right quaternionic Hilbert space with inner product
\begin{equation}
\label{3.2}
\langle f, \, g\rangle=\sum_{k=0}^\infty \bar g_k f_k\quad\mbox{if}\quad
f(p)=\displaystyle\sum_{k=0}^\infty p^kf_k, \; \, g(p)=\sum_{k=0}^\infty p^kg_k.
\end{equation}
A power-series computation followed by integration of (uniformly converging on compact sets) power
series  shows that for $f$ as in \eqref{3.2} and for a fixed $I\in\mathbb S$,
\begin{align*}
\int_{0}^{2\pi}|f(re^{I\theta})|^2d\theta&=
\int_{0}^{2\pi}\left(\sum_{j,k=0}^\infty r^{k+j}\overline{f}_ke^{I(j-k)\theta}f_j\right)d\theta\\
&=\sum_{j,k=0}^\infty r^{k+j}\overline{f}_k
\left(\int_{0}^{2\pi}e^{I(j-k)\theta}\theta \right)f_j=2\pi \cdot\sum_{n=0}^\infty r^{2n}|f_n|^2.
\end{align*}
The latter formula implies that the norm in $\rH$ can be equivalently defined as
\begin{equation}
\| f\|^2_{\rH}= \sup_{0\le r< 1} \frac{1}{2\pi}\int_0^{2\pi} | f(re^{I\theta})|^2 d\theta
\label{3.3}
\end{equation}
where the value of the integral on the right is the same for each $I\in\mathbb S$.
Observe that the supremum in the last formula can be replaced by the limit as $r$ tends to one.

\smallskip

The space $\rH$ can be alternatively characterized as the reproducing
kernel Hilbert space with reproducing kernel
\begin{equation}
\label{3.4}
k_{\rH}(p,q)=\sum_{n=0}^\infty p^n \overline{q}^n.
\end{equation}
The latter means that the function $k_{\rH}(\cdot,q)$ belongs to $\rH$ for every
$q\in\B$ and for any function $f\in\rH$ as in \eqref{3.2},
\begin{equation}
\left\langle f, \, k_{\rH}(\cdot,q) \right\rangle_{\rH}=\sum_{k=0}^\infty q^k f_k=f(q).
\label{3.5}
\end{equation}
\begin{proposition}
The finite collection of functions $\{k_{\rH}(\cdot,q_i)\}$ based on distinct points
$q_1,\ldots,q_k\in\mathbb B$
is (right) linearly independent in $\rH$ if and only if none three of these points belong to the
same $2$-sphere.
\label{R:3.0}
\end{proposition}
\begin{proof}
Let us assume that
$\sum_{i=1}^kk_{\rH}(p,q_i)\alpha_i\equiv 0$. Substituting the
power series expansions \eqref{3.4} for $k_\rH(\cdot,q_i)$ into
this identity and equating the corresponding coefficients we
conclude that the columns of the Vandermonde matrix
$V=\left[q_i^{j-1}\right]_{i,j=1}^k$ are linearly dependent which
is the case  if and only if there are three equivalent points in
$\{q_1,\ldots,q_k\}$; we refer to \cite{lam_paper} for more details.
\end{proof}
\begin{remark}
{\rm The linear dependence of three functions $k_{\rH}(\cdot,p_i)$ based on equivalent points
implies that
the restriction of any function $f\in\rH$ to any $2$-sphere is completely determined by the values
of $f$ at any two
points of this sphere. Indeed, if
$$
p_i=x+yI_i,\quad (x,y\in{\mathbb R}, \; I_i\in{\mathbb S}, \; i=1,2,3),
$$
then it is readily checked that
$$
\overline{p}_3^n=\overline{p}_1^n(I_1-I_2)^{-1}(I_3-I_2)+\overline{p}_2^n(I_1-I_2)^{-1}(I_1-I_3)
\quad\mbox{for all}\quad n\ge 0
$$
which implies the identity
\begin{equation}\label{3.6}
k_{\rH}(p,p_3)\equiv k_{\rH}(p,p_1)(I_1-I_2)^{-1}(I_3-I_2)+k_{\rH}(p,p_2)(I_1-I_2)^{-1}(I_1-I_3).
\end{equation}
Combining the latter identity with \eqref{3.5} leads us to
\begin{align}
f(p_3)=\left\langle f, \, k_{\rH}(\cdot,p_3) \right\rangle_{\rH}=&
\left\langle f, \, k_{\rH}(\cdot,p_1)(I_1-I_2)^{-1}(I_3-I_2) \right\rangle_{\rH}\notag\\
&+\left\langle f, \, k_{\rH}(\cdot,p_2)(I_1-I_2)^{-1}(I_1-I_3) \right\rangle_{\rH}\notag\\
=&(\overline{I}_1-\overline{I}_2)^{-1}(\overline{I}_3-\overline{I}_2)f(p_1)+
(\overline{I}_1-\overline{I}_2)^{-1}(\overline{I}_1-\overline{I}_3)f(p_2)\notag\\
=&(I_2-I_1)^{-1}\left\{(I_2-I_3)f(p_1)+(I_3-I_1)f(p_2)\right\}.\label{3.7}
\end{align}}
\label{R:pr}
\end{remark}
The latter representation was established in  \cite{cgss} for
general slice regular functions on axially symmetric s-domains.  We now pass to the main result of
this section.
\begin{theorem}
Let $S: \, \mathbb B\to \mathbb H$. The following are equivalent:
\begin{enumerate}
\item $S$ is slice regular on $\mathbb B$ and $|S(p)|\le 1$ for
all $p\in\mathbb B$.
\item The operator $M_S$ of  left $\star$--multiplication by $S$
\begin{equation}\label{3.8}
M_S: \, f\mapsto S\star f
\end{equation}
is a contraction on $\rH$, that is, $\|S\star f\|_{\rH}\le \|f\|_{\rH}$ for all $f\in\rH$.
\item The kernel
\begin{equation}\label{3.9}
K_S(p,q)=\sum_{k=0}^\infty p^k(1-S(p)\overline{S(q)})\bar q^k
\end{equation}
is positive on $\mathbb B\times\mathbb B$.
\item $S\in\mathcal R(\mathbb B)$ and ${\bf I}_n-{\bf S}_n{\bf
S}_n^*\ge 0$
for all $n\ge 0$ where ${\bf S}_n$ is the matrix given in \eqref{1.4}.
\end{enumerate}
\label{T:3.1}
\end{theorem}
\begin{proof} We first remark that the operator $M_S$ can be defined via formula \eqref{2.2}
which does not assume any regularity of $S$. However, if $M_S$ maps $\rH$ into itself, then
the function $S=M_S 1$ belongs to $\rH$ and hence is slice regular.

\smallskip

{\em Proof of  $(2)\Longrightarrow(3)$}: Let us assume that $M_S: \, \rH\to \rH$ is a
contraction. Combining formulas \eqref{2.2} and \eqref{3.4} gives
$$
M_Sk_{\rH}(\cdot,q)=\sum_{j=0}^\infty p^kS(p)\bar q^k
$$
which together with reproducing kernel property \eqref{3.5} implies
\begin{align}
\left(M_S^*k_{\rH}(\cdot,q)\right)(p)&=\left\langle M_S^*k_{\rH}(\cdot,q), \, k_{\rH}(\cdot,p)
\right\rangle_{\rH}\notag\\
&=\left\langle k_{\rH}(\cdot,q), \, S\star
k_{\rH}(\cdot,p)\right\rangle_{\rH}=
\sum_{k=0}^\infty p^k\overline{S(q)}\bar q^k,
\label{3.10}
\end{align}
and subsequently,
$$
\left\langle (I-M_SM_S^*)k_{\rH}(\cdot,q), \, k_{\rH}(\cdot,p)\right\rangle_{\rH}=
\sum_{k=0}^\infty p^k(1-S(p)\overline{S(q)})\bar q^k.
$$
Therefore, for any function $f\in\rH$ of the form
\begin{equation}
\label{3.11}
f=\sum_{i=1}^r k_{\rH}(\cdot, p_i)\alpha_i,\qquad
r\in{\mathbb N}, \; p_i\in\mathbb B, \; \alpha_i\in\mathbb H,
\end{equation}
we have
\begin{align}
\left\langle (I-M_SM_S^*)f, \, f\right\rangle_{\rH}&=
\left\langle f, \, f\right\rangle_{\rH}-
\left\langle M_S^*f, \, M_S^*f\right\rangle_{\rH}\notag \\
&=\sum_{i,j=1}^r \overline{\alpha}_i k_{\rH}(p_i,p_j)\alpha_j
-\sum_{i,j=1}^r\sum_{k=0}^\infty \overline{\alpha}_i
p_i^k S(p_i)\overline{S(p_j)}\overline{p}_i^k\alpha_j\notag \\
&=\sum_{i,j=1}^r \overline{\alpha}_i K_S(p_i,p_j)\alpha_j.\label{3.12}
\end{align}
Since $M_S: \, \rH\to \rH$ is a contraction,
the inner product on the left hand side of \eqref{3.12} is nonnegative.
Consequently, the quadratic  form on the right hand side of
\eqref{3.12} is  nonnegative so that $K_S$ is a positive kernel.

\smallskip

{\em Proof of  $(3)\Longrightarrow(2)$}: Let us assume that the kernel \eqref{3.9} is positive on
$\mathbb B\times \mathbb B$. Observing that the function on the right side of \eqref{3.10}
belongs to $\rH$ (for each fixed $q\in\mathbb B$) with $\rH$-norm equal
$\frac{|S(q)|^2}{1-|q|^2}$, we define the operator $T: \, \rH\to \rH$ by letting
$T: \, k_{\rH}(\cdot,q)\mapsto {\displaystyle \sum_{k=0}^\infty p^k\overline{S(q)}\bar q^k}$
(that is, using the formula \eqref{3.8} obtained earlier for $M_S^*$) with subsequent extention
by linearity to functions $f$ of the form \eqref{3.11} and then, since such functions are dense in
$\rH$, extending by continuity to all of $\rH$. Due to this density, the calculation \eqref{3.12}
(with $T$ instead of $M_S^*$) shows that $T$ is a contraction on $\rH$. We then calculate
its adjoint getting $T^*f=S\star f=M_Sf$. Since $T$ is a contraction on $\rH$, its adjoint $M_S$ is
a
contraction as well.

\smallskip

{\em Proof of  $(3)\Longrightarrow(1)$}: If the kernel $K_S$ is positive on $\mathbb B\times\mathbb
B$, we have, in particular,
$$
0\le K_S(q,q)=\sum_{k=0}^\infty q^k(1-|S(q)|^2)\bar q^k=\frac{1-|S(q)|^2}{1-|q|^2}
$$
and therefore, $|S(q)|\le 1$ for every $q\in\mathbb B$. On the other hand, by implication
$(3)\Longrightarrow(2)$, the operator $M_S$ maps $\rH$ into itself and thus $S=M_S1$
belongs to $\rH\subset \mathcal R(\mathbb B)$.

\smallskip

{\em Proof of  $(1)\Longrightarrow(2)$}:  We now assume that $S$ is in
$\mathcal R(\mathbb  B,\overline{\mathbb B})$, i.e., $S$ is
slice regular and with $|S(p)|\le 1$ for all $p\in\mathbb B$.
By formulas \eqref{3.3} and \eqref{2.3}, we have for every $f\in\rH$ and every $I\in\mathbb S$,
\begin{align}
\| f\star S\|^2_{\rH}&=\sup_{0\le r< 1} \frac{1}{2\pi}\int_0^{2\pi}|f\star S(re^{I\theta})|^2
d\theta\notag \\
&= \sup_{0\le r< 1} \frac{1}{2\pi}\int_0^{2\pi}| f(re^{I\theta})S(f(re^{I\theta})^{-1}
re^{I\theta}f(re^{I\theta})|^2 d\theta\notag \\
& \leq \sup_{0\le r<1}\frac{1}{2\pi}\int_0^{2\pi} |f(r e^{I\theta})|^2 d\theta= \|
f\|^2_{\rH}.\label{3.13}
\end{align}
Let $S^c$ and $f^c$ be the slice conjugates of $S$ and $f$ defined via formula \eqref{2.5}. Due to
obvious relations  $(f^c)^c=f, \quad \|f\|_{\rH}=\|f^c\|_{\rH},\quad (f\star g)^c=g^c\star f^c$
holding for
all $f,g\in\rH$, we have from \eqref{3.13}
\begin{equation}
\|S^c\star f^c\|_{\rH}=\|(f\star S)^c\|_{\rH}=\|f\star S\|_{\rH}\le \| f\|_{\rH}=\| f^c\|_{\rH}.
\label{3.14}
\end{equation}
Therefore the operator $M_{S^c}: \, f\mapsto S^c\star f$ is a contraction on $\rH$. By implications
$(2)\Longrightarrow(3)\Longrightarrow(1)$ which have been already proved, we conclude that $S^c\in\mathcal
R(\mathbb  B,\overline{\mathbb B})$.
We then apply \eqref{3.14} to $S^c$ rather than to $S$ concluding that the operator $M_{(S^c)^c}=M_S$
is a contraction on $\rH$.

\smallskip

{\em Proof of  $(2)\Longrightarrow(4)$}: The proof is similar to that of $(2)\Longleftrightarrow(4)$
with the only difference that
instead of functions of the form \eqref{3.11} we will use slice regular polynomials, namely the polynomials with coefficients written on the right.
We first assume
(2). The calculation analogous to that
in \eqref{3.10} shows that for $S$ with the Taylor series as in \eqref{1.4},
$$
M_S^*: \, p^k\mapsto \sum_{j=0}^k p^j\overline{S}_{k-j}\quad\mbox{for all}\quad k\ge 0
$$
which extends by linearity to
$$
M_S^*: \, f(p)=\sum_{k=0}^n p^k f_k\mapsto \sum_{k=0}^np^k\left(
\sum_{j=k}^{n}\overline{S}_{j-k}f_j\right).
$$
Letting ${\bf f}:=\begin{bmatrix} f_0 & f_1 & \ldots & f_n\end{bmatrix}^\top$
we  get the following analog of \eqref{3.12} in terms of the matrix ${\bf S}_n$ from \eqref{1.4}:
\begin{equation}
\|f\|_{\rH}^2-\|M_S^*f\|_{\rH}^2=\sum_{k=0}^n
|f_k|^2-\sum_{k=0}^n\left|\sum_{j=k}^{n}\overline{S}_{j-k}f_j \right|^2
={\bf f}^*\left({\bf  I}_n-{\bf S}_n{\bf S}_n^*\right){\bf f}.
\label{3.15}
\end{equation}
If $M_S$ is a contraction on $\rH$, the latter expression is nonnegative for every vector ${\bf
f}\in\mathbb H^{n+1}$
and therefore the matrix ${\bf  I}_n-{\bf S}_n{\bf S}_n^*$ is positive semidefinite. Conversely, if
this matrix is positive
semidefinite for each $n\ge 1$, then equality \eqref{3.15} shows that  $M_S^*$ acts contractively
(in $\rH$-metric) on any
polynomial. Since the polynomials are dense in $\rH$, the operators $M_S^*$ and $M_S$ are
contractions on the whole $\rH$.
\end{proof}
We point out several consequences of the last theorem.
\begin{corollary}
Let $S: \,V\to\mathbb H$ be such that the kernel \eqref{3.9} is positive on $V\times V$, where
$V$ is an open subset of $\mathbb B$. Then $S$ extends to a function from $\mathcal R(\mathbb
B,\overline{\mathbb B})$.
\label{R:3.2}
\end{corollary}
\begin{proof}
The proof is the same as that of implication $(3)\Longrightarrow(2)$ in Theorem \ref{T:3.1} once we
observe
that the functions of the form \eqref{3.11} with $p_i\in V$ (rather than in $\mathbb B$) are still
dense in $\rH$.
\end{proof}
In analogy to the classical case we may introduce the Hardy space ${\rm H}^\infty(\mathbb B)$ of
bounded slice regular
functions on $\mathbb B$ with norm $\|S\|_{\infty}=\sup_{p\in\mathbb B}|S(p)|<\infty$ and the space
${\mathcal M}(\rH)$ of
bounded multipliers, that is, the functions $S: \, \mathbb B\to\mathbb H$ such that the operator
$M_S$ of left
$\star$--multiplication \eqref{3.8} is bounded on $\rH$. By the very definition,
$\mathcal R(\mathbb  B,\overline{\mathbb B})$ is the closed
unit ball of ${\rm H}^\infty(\mathbb B)$. The following conclusion is a consequence of Theorem
\ref{T:3.1}.
\begin{corollary}
${\rm H}^\infty(\mathbb B)={\mathcal M}(\rH)$ and $\|S\|_\infty =\|M_S\|$ for every
$S\in {\rm H}^\infty(\mathbb B)$.
\label{R:3.3}
\end{corollary}
\begin{proof}
If $S\in {\rm H}^\infty(\mathbb B)$ with  $\|S\|_{\infty}=r>0$, then the scaled function
$\frac{1}{r}S$ belongs to $\mathcal R(\mathbb  B,\overline{\mathbb B})$
and by Theorem \ref{T:3.1}, the operator $M_{\frac{1}{r}S}: \, \rH\to \rH$
is a contraction. Since $M_{\frac{1}{r}S}=\frac{1}{r}M_S$, we conclude that $\|M_S\|=\|r
M_{\frac{1}{r}S}\|\le r$ and  thus,
$S\in {\mathcal M}(\rH)$ with $\|M_S\|\le \|S\|_\infty$. In particular, ${\rm H}^\infty(\mathbb
B)\subset{\mathcal M}(\rH)$.
The reverse inclusion and the reverse norm inequality is established in much the same way.
\end{proof}
As another consequence of Theorem \ref{T:3.1}, we get the necessity part in Theorem \ref{T:1.2}.
\begin{corollary}
If the problem $\PR$ has a solution, then the Pick matrix $P$ is positive semidefinite.
\label{C:3.6}
\end{corollary}
\begin{proof}
Let $S$ be a solution to the problem $\PR$. Since $S\in\mathcal R(\mathbb B,\overline{\mathbb B})$,
the kernel $K_S$ is positive on $\mathbb B\times \mathbb B$. Then the matrix
$\left[K_S(p_i,p_j)\right]_{i,j=1}^n$ is positive
semidefinite. Since $S$ satisfies the interpolation conditions \eqref{1.5},
\begin{equation}
K_S(p_i,p_j)=\sum_{k=0}^\infty p_i^k(1-S(p_i)\overline{S(p_j)})\overline{p}_j^k=\sum_{k=0}^\infty
p_i^k(1-s_i\overline{s}_j)\overline{p}_j^k.
\label{3.16}
\end{equation}
Comparing \eqref{3.16} with \eqref{1.6} we see that the matrix $\left[K_S(p_i,p_j)\right]_{i,j=1}^n$
is equal to the Pick matrix $P$
which is therefore positive semidefinite.
\end{proof}

\section{Characterization of solutions to the problem $\PR$ in terms of positive kernels}
\setcounter{equation}{0}

The classical complex-valued Nevanlinna-Pick
problem has been studied using different approaches, including in particular, the iterative Schur algorithm
\cite{nevan1}, the Commutatnt Lifting approach \cite{sarason}, the
Grassmanian approach \cite{ballhelton}, Potapov's method of fundamental matrix
inequalities \cite{potap} and its far-reaching extension the Abstract Interpolation Problem approach \cite{kky, khyu}.
Each method has its strengths and weaknesses; so it would be interesting to clarify how each of
them  extends to the quaternionic setting. The method we chose
for the present paper has its origins in \cite{potap}. The
first step is carried out in the next  theorem which characterizes
solutions of the problem $\PR$ in terms of positive kernels of special structure.

\begin{theorem}
A function  $S:\B\to {\mathbb H}$ is a solution to the problem $\PR$ if and only if the
following kernel is positive on $\B\times \B$:
\begin{equation}
\widehat K_S(p,q):=\begin{bmatrix} P & B^S(q)^* \\  B^S(p) & K_S(p,q) \end{bmatrix}\succeq 0,
\label{4.1}
\end{equation}
where $P$ is given in \eqref{1.6} and where
\begin{align}
B^S(p)&=\begin{bmatrix} B^S_1(p) & B^S_2(p) & \ldots & B^S_n(p)\end{bmatrix}\label{4.2}\\
&=\sum_{k=0}^\infty p^k\begin{bmatrix} (1-S(p)\bar s_1) \bar p_1^k &
(1-S(p)\bar s_2) \bar p_2^k  & \ldots & (1-S(p)\bar s_n) \bar p_n^k\end{bmatrix}.\notag
\end{align}
\label{T:2.3}
\end{theorem}
\begin{proof}
For the necessity part we modify the argument used in the previous section. If $S$ is a solution to
the problem $\PR$,
it belongs to $\mathcal R(\mathbb  B,\overline{\mathbb B})$ and therefore, the kernel $K_S$ is positive on
$\mathbb B\times \mathbb
B$.
Let us pick finitely many points $q_1,\ldots,q_r\in \mathbb B$ and let us consider the (positive
semidefinite) matrix
$R=\left[K_S(\zeta_i,\zeta_j)\right]_{i,j=1}^{(n+1)r}$ based on the $(n+1)r$ points $\zeta_j$ chosen
as follows:
$$
\zeta_j=\left\{\begin{array}{lll}p_i &{\rm if} & j=i \, {\rm mod}(n+1),\\ q_\ell &{\rm if} &
j=(n+1)\ell.
\end{array}\right.
$$
Since $S$ satisfies interpolation conditions \eqref{1.5}, the entries $K_S(p_i,p_j)$ in $R$ are
given as in \eqref{3.16}.
On the other hand, in view of \eqref{4.2},
\begin{equation}
K_S(q_i,p_j)=\sum_{k=0}^\infty q_i^k(1-S(q_i)\overline{S(p_j)})\overline{p}_j^k=\sum_{k=0}^\infty
q_i^k(1-S_i\overline{s}_j)\overline{p}_j^k=B_j^S(q_i).
\label{4.3}
\end{equation}
A careful but straightforward verification based on \eqref{3.16} and \eqref{4.3} confirms that
the matrix $R=\left[K_S(\zeta_i,\zeta_j)\right]_{i,j=1}^{(n+1)r}$ can be written in the block-matrix
form as
$$
R=\left[\widehat K_S(q_i,q_j)\right]_{i,j=1}^r
$$
where $\widehat K$ is the kernel defined in \eqref{4.1}. Since $R$ is positive semidefinite and the
points $q_1,\ldots,r$ were chosen
arbitrarily in $\mathbb B$, the kernel \eqref{4.1} is positive on $\mathbb B\times \mathbb B$.

\smallskip

The proof of the sufficiency part is based on the operator-theoretic argument involving Schur
complements and multiplication operators which is adapted from \cite{bb, bol}.
Let us assume that the kernel \eqref{4.1} is positive on $\B\times \B$. Then in
particular, the kernel $K_S$ is positive on  $\B\times \B$ and therefore, $S\in\mathcal R(\mathbb
B,\overline{\mathbb B})$.  Furthermore,  it
follows from \eqref{4.1} that the following $2\times 2$ matrix valued  kernel is positive
\begin{equation}
K_i(p,q)=\begin{bmatrix} P_{ii} & B^S_i(q)^* \\ B^S_i(p) & K_S(p,q)\end{bmatrix}\succeq 0
\label{4.4}
\end{equation}
for each $i=1,\ldots,n$.
The positivity condition \eqref{4.4} is equivalent to the positivity of
the operator
\begin{equation}
{\bf P}_i=\begin{bmatrix} P_{ii} & M_{B^S_i}^* \\  M_{B^S_i} & I-M_SM_S^*\end{bmatrix}: \;
\begin{bmatrix}\mathbb H \\ \rH \end{bmatrix}\to \begin{bmatrix}\mathbb H \\ \rH \end{bmatrix}
\label{4.5}
\end{equation}
due to the identity
$$
\left\langle {\bf P}_i
\begin{bmatrix}\alpha \\ k_{\rH}(\cdot,q)\beta \end{bmatrix}, \; \begin{bmatrix}\alpha^\prime \\
 k_{\rH}(\cdot,p)\beta^\prime\end{bmatrix}
\right\rangle_{\mathbb H\oplus\rH}
=\left\langle K_i(p,q)
\begin{bmatrix}\alpha \\ \beta\end{bmatrix}, \; \begin{bmatrix}\alpha^\prime \\
\beta^\prime\end{bmatrix}
\right\rangle_{\mathbb H^2}
$$
holding for all $\alpha,\alpha^\prime,\beta,\beta^\prime\in\mathbb H$ and all $p,q\in\mathbb B$, and
since
linear combinations of vectors of the form $\alpha\oplus k_{\rH}(\cdot,q)\beta$
($\alpha,\beta\in{\mathbb H}, \, q\in\mathbb B$) are dense in $\mathbb H\oplus\rH$.
%Note that all the computations are allowed since it is immediate that ${\rm H}^\infty\subset \rH$.
\smallskip

We next fix $i\in\{1,\ldots,n\}$ and introduce two operators $T_1, \, T_2: \,
\mathbb H\to \rH$ as follows:
\begin{equation}
T_1\alpha=k_{\rH}(\cdot,p_i)\alpha\quad\mbox{and}\quad T_2\alpha =:\begin{cases} \,
k_{\rH}(\cdot,s_i^{-1}p_is_i)\bar s_i\alpha,\,\,{\rm if}\,\, s_i\not=0,\\
\; 0 ,\quad\quad \hspace{2.1cm} {\rm if}\,\,s_i=0.
\end{cases}
\label{4.6}
\end{equation}
Since $k_{\rH}$ is the reproducing kernel for $\rH$ we have
$$
T_1^*T_1-T_2^*T_2=\left\{\begin{array}{lll}k_{\rH}(p_i,p_i)-s_ik_{\rH}(s_i^{-1}p_is_i,s_i^{-1}p_is_i)\bar
s_i=\frac{1-|s_i|^2}{1-|p_i|^2} &\mbox{if}& s_i\neq 0, \\
k_{\rH}(p_i,p_i)=\frac{1}{1-|p_i|^2}&\mbox{if}& s_i=0,\end{array}\right.
$$
which being compared with \eqref{1.7} gives
$$
T_1^*T_1-T_2^*T_2=P_{ii}.
$$
We also observe from \eqref{4.2} that the function $B_i^S$ can be written as
$$
B_i^S=k_{\rH}(\cdot ,p_i)-S\star k_{\rH}(\cdot,s_i^{-1}p_is_i)\bar s_i,
$$
so that $M_{B^S_i}=T_1-M_ST_2$. Therefore, we can rewrite \eqref{4.5} as
$$
{\bf P}_i=\begin{bmatrix} T_1^*T_1-T_2^*T_2 & T_1^*-T_2^*M_S
\\ T_1-M_ST_2 & I-M_SM_S^*\end{bmatrix}.
$$
The operator ${\bf P}_i$ equals the Schur complement of the left top block
in the extended operator
$$
\widehat {\bf P}_i=\begin{bmatrix} I & T_2 & M_S^* \\ T_2^* & T_1^*T_1 & T_1^*\\
M_S & T_1 & I\end{bmatrix}: \, \begin{bmatrix}\rH \\ \mathbb H \\ \rH\end{bmatrix}\to
\begin{bmatrix}\rH \\ \mathbb H \\ \rH\end{bmatrix},
$$
and therefore, ${\bf P}_i\ge 0$ if and only if $\widehat {\bf P}_i\ge 0$. But then
the Schur complement of the right bottom block in $\widehat {\bf P}_i$ is also positive
semidefinite:
$$
\begin{bmatrix}I-M_S^*M_S & T_2-M_S^*T_1 \\ T_2^*-T_1^*M_S & T_1^*T_1-T_1^*T_1\end{bmatrix}
=\begin{bmatrix}I-M_S^*M_S & T_2-M_S^*T_1 \\ T_2^*-T_1^*M_S & 0 \end{bmatrix}\ge 0
$$
from which we conclude $T_2-M_S^*T_1=0$. We next use \eqref{3.10} and definitions
\eqref{4.6} to rewrite the last equality as
$$
0\equiv T_2-M_S^*T_1=k_{\rH}(p,s_i^{-1}p_is_i)\bar s_i-\sum_{k=0}^\infty p^k\overline{S(p_i)}\bar
p_i^k
$$
and finally, letting $p=0$ we get $\bar s_i=\overline{S(p_i)}$ which is equivalent to \eqref{1.5}.
Thus, $S$ solves the problem $\PR$.
\end{proof}
\begin{remark}
{\rm The positivity condition \eqref{4.1} implies $P\ge 0$; thus Theorem \ref{T:2.3} contains the
necessity part
of Theorem \ref{T:1.2}.}
\label{R:4.2}
\end{remark}

\begin{remark}
{\rm The Pick matrix $P$ of the problem $\PR$ satisfies the Stein equality
\begin{equation}
P-TPT^*=EE^*-NN^*
\label{4.8}
\end{equation}
where
\begin{equation}
T=\left[\begin{array}{ccc}p_1 &&0
\\ &\ddots & \\ 0 && p_n\end{array}\right],\quad
E=\left[\begin{array}{c}1\\ \vdots \\ 1 \end{array}\right],\quad
N=\left[\begin{array}{c}s_1 \\ \vdots\\ s_n\end{array}\right].
\label{4.9}
\end{equation}
The entry-wise verification of \eqref{4.8} is immediate. In fact, if $T$ is any square matrix
with right spectrum contained in $\B$, then the Stein equation $P-TPT^*=D$ has a unique solution
given by converging series  $P=\sum_{k\ge 0}T^kDT^{*k}$. In particular,
if $D=EE^*-NN^*$, this series produces $P$ as in \eqref{1.6}.}
\label{R:4.4}
\end{remark}
\begin{remark}
{\rm Let us note that the function \eqref{4.2} can be written in terms of \eqref{4.9}  as
\begin{equation}
B^S(p)=\sum_{k=0}^\infty p^k\left (E^*-S(p)N^*\right)T^{*k}
=\begin{bmatrix}1 & -S(p)\end{bmatrix}\star\left (\sum_{k=0}^\infty p^k\begin{bmatrix}E^* \\
N^*\end{bmatrix}T^{*k}\right)
\label{4.10}
\end{equation}
Therefore, all the entries in the kernel inequality \eqref{4.1} are defined in terms
of given $E$, $N$, $T$ and an unknown function $S$. The description of all functions $S$
satisfying the latter inequality does not rely on the specific formulas \eqref{1.6}, \eqref{4.9};
it will be established under the assumptions that (1) the right spectrum of $T$ is contained in $\B$
and
(2) the unique solution $P$ of the Stein equation  \eqref{4.8} is positive semidefinite.}
\label{R:2.4}
\end{remark}
We conclude this section with two results which substantially
simplify the subsequent analysis. The first one is about the
"consistency" of interpolation data set.
\begin{lemma}
Let us assume that the Pick matrix $P$ \eqref{1.6} is positive
semidefinite and that three interpolation nodes, say $p_1$, $p_2$
and $p_3$ belong to the same $2$-sphere:
\begin{equation}
p_i=x+yI_i,\quad (x,y\in{\mathbb R}, \; I_i\in{\mathbb S}, \; i=1,2,3).
\label{4.11}
\end{equation}
Then the three top rows in $P$ are left linearly dependent and the target values $s_1$, $s_2$ and
$s_3$
are related by
\begin{equation}
s_3=(I_2-I_1)^{-1}\left\{(I_2-I_3)s_1+(I_3-I_1)s_2\right\}.
\label{4.12}
\end{equation}
\label{L:3.6}
\end{lemma}
\begin{proof}
Let us define two positive semidefinite  matrices
$$
P_1={\displaystyle\sum_{k=0}^\infty T^kEE^*T^{*k}}\quad\mbox{and}\quad
P_2={\displaystyle\sum_{k=0}^\infty T^kNN^*T^{*k}}
$$
and observe that
\begin{equation}
P_1=TP_1T^*+EE^*, \quad P_2=TP_2T^*+NN^*\quad\mbox{and}\quad P=P_1-P_2.
\label{4.13}
\end{equation}
The matrix $P_1$ can be written more explicitly as
$$
P_1=\left[k_{\rH}(p_i,p_j)\right]_{i,j=1}^n=\left[\left\langle k_{\rH}(\cdot, p_j), \,
k_{\rH}(\cdot, p_i)\right\rangle\right]_{i,j=1}^n
$$
and is, therefore, the gram matrix of the set $\{k_{\rH}(\cdot, p_i)\}_{i=1}^n$. Due to identity
\eqref{3.6}, the three top rows in $P$
are
left linearly dependent and moreover,
$$
{\bf x}P_1=0,\quad\mbox{where}\quad {\bf x}=\begin{bmatrix}(I_1-I_2)^{-1}(I_2-I_3) &
(I_1-I_2)^{-1}(I_3-I_1)& 1 & 0 & \ldots & 0\end{bmatrix}.
$$
Since $P=P_1-P_2\ge 0$, it also follows that ${\bf x}P_2=0$ and therefore, by the second relation in
\eqref{4.13}, ${\bf x}N=0$.
Substituting explicit formulas for ${\bf x}$ and $N$ into the latter equality gives \eqref{4.12}.
\end{proof}
Comparing \eqref{4.12} and \eqref{3.7} shows that condition $P\ge 0$ indeed guarantees that the
target value $s_3$ at $p_2$ for the
unknown
slice regular interpolant is consistent with its values at $p_1$ and $p_2$. This advances us toward
establishing the "if" part in Theorem
\ref{T:1.2}: now it suffices to prove Theorem \ref{T:1.2} under the assumption  ${\bf (A)}$.
\begin{lemma}
Let us assume that ${\bf (A)}$ holds and the Pick matrix $P\ge 0$ is singular. Then the problem
$\PR$ has at most one solution which, if
exists, is given by the formula
\begin{equation}
S(p)=R\star Q(p)^{-\star},
\label{4.14}
\end{equation}
where
\begin{equation}
R=\sum_{i=1}^nk_{\rH}(\cdot,p_i)\alpha_i,\quad Q=\sum_{i: s_i\neq 0}k_{\rH}(\cdot,s_i^{-1}
p_i s_i)\overline{s}_i\alpha_i
\label{4.15}
\end{equation}
and where ${\bf y}=\left[\begin{smallmatrix} \alpha_1 \\
\vdots \\ \alpha_n\end{smallmatrix}\right]\in{\mathbb H}^n$ is any nonzero vector such that $P{\bf
y}=0$.
\label{L:4.7}
\end{lemma}
\begin{proof}
Let us assume that $S$ is a solution to the problem $\PR$. Then the matrix
$$\begin{bmatrix} P & B^S(p)^* \\  B^S(p) & \frac{1-|S(p)|^2}{1-|p|^2}\end{bmatrix}\ge 0
\quad\mbox{for all}\quad p\in\mathbb B,
$$
is positive
semidefinite for every $p\in\mathbb B$.
From this positivity and from the equality $P{\bf y}=0$ we conclude that $B^S(p){\bf y}\equiv 0$. Making
use of the formula \eqref{4.2} for $B^S$,
we write the latter identity  more explicitly as
\begin{align*}
0\equiv \sum_{i=1}^n\sum_{k=0}^\infty p^k\left(1-S(p)\overline{s}_i\right)\overline{p}_i^k
\alpha_i&=
\sum_{i=0}^n k_{\rH}(p,p_i)\alpha_i -\sum_{s_i\neq 0}\sum_{k=0}^\infty
p^kS(p)\left(\overline{s_i^{-1}p_is_i}\right)^k\overline{s}_i\alpha_i\\
&=R(p)-S\star Q(p)
\end{align*}
where the last step follows by formulas \eqref{4.15} and the definition of the $\star$-product. Thus
any solution $S$ to the problem
$\PR$ must satisfy
\begin{equation}
S\star Q(p)=R(p)\quad\mbox{for all}\quad p\in{\mathbb B}.
\label{4.16}
\end{equation}
By Proposition \ref{R:3.0} and due to assumption $({\bf A})$, the
function $R$ is not vanishing identically. Then it follows from
\eqref{4.16} that $Q$ is not vanishing identically as well.
Therefore, the formula \eqref{4.14} holds (first on an open
subset of $\mathbb B$ and then by continuity on the whole
$\mathbb B$, since $S$ is assumed to be in $\mathcal R(\mathbb B,\overline{\mathbb B})$).
So the solution (if exists) is unique, and this uniqueness
implies in particular, that the representation \eqref{4.14} does
not depend on the particular choice of ${\bf y}\in{\rm Ker} \, P$.
\end{proof}
It seems tempting to verify directly that the function $S$
defined in \eqref{4.14} belongs to $\mathcal R(\mathbb B,\overline{\mathbb B})$
and satisfies interpolation conditions \eqref{1.5}. The first part can be
achieved easily using the extension arguments (similar to those
used in the proof of Theorem \ref{T:2.6} below). Verification of
interpolation equalities is much harder, so the existence part
will be proven in Section 6 using the reduction method.

\section{The indeterminate case}
\setcounter{equation}{0}

In this section we handle the case where the Pick matrix $P$ of the problem $\PR$ is positive
definite.
By Lemma \ref{L:3.6} this may occur only if none three of the interpolation nodes belong to the same
$2$-sphere. On the other hand, Lemma \ref{L:4.7} tells us that this is the only option for the
indeterminacy.
We will show that in this case the problem  $\PR$ indeed has infinitely many solutions and we will
describe all solutions in terms of a linear fractional formula. Thus, assuming that $P$ is positive
definite
and making use of notation \eqref{4.9}, we introduce the $2\times 2$ matrix-valued function
\begin{equation}
\Theta(p)={\bf I}_2+(p-1)\sum_{k=0}^\infty p^k \begin{bmatrix}E^* \\ N^*\end{bmatrix}T^{*k}
P^{-1}({\bf I}_n-T)^{-1}\begin{bmatrix}
E & -N\end{bmatrix}\label{5.1}
\end{equation}
which is clearly slice regular in $\B$.
\begin{proposition}
Under assumptions (\ref{4.8}) and (\ref{4.9}) in Remark \ref{R:4.4}, let $\Theta$ be defined by
formula \eqref{5.1} and let
\begin{equation}
\Theta(z)=\begin{bmatrix}\Theta_{11}(p) & \Theta_{12}(p) \\ \Theta_{21}(p)&
\Theta_{22}(p)\end{bmatrix}
\quad\mbox{and}\quad J=\begin{bmatrix}1 & 0 \\ 0 & -1\end{bmatrix}.
\label{5.2}
\end{equation}
Then the kernel
\begin{equation}
K_{\Theta,J}(p,q)=\sum_{k=0}^\infty p^k\left(J-\Theta(p)J\Theta(q)^*\right)\bar q^k
\label{5.3}
\end{equation}
is positive on $\B\times\B$. Furthermore, $|\Theta_{22}(p)|>1$ for every $p\in\B$ and the functions
$\Theta_{22}^{-\star}$ and $\Theta_{22}^{-\star}\star \Theta_{21}$ are both in $\mathcal R(\mathbb
B,\overline{\mathbb B})$.
\label{P:2.5}
\end{proposition}
\begin{proof} A straightforward computation relying  solely on the identity \eqref{4.9}
shows that
\begin{equation}
K_{\Theta,J}(p,q)=\left(\sum_{k=0}^\infty p^k \begin{bmatrix}E^* \\ N^*\end{bmatrix}T^{*k}\right)
P^{-1}\left(\sum_{k=0}^\infty T^{k} \begin{bmatrix}E & N\end{bmatrix}\bar q^k\right)
\label{5.4}
\end{equation}
from which the positivity of $K_{\Theta,J}$ follows. The bottom diagonal entry
$K^{22}_{\Theta,J}$ of this kernel equals (as is easily seen from \eqref{5.2} and
\eqref{5.3})
\begin{equation}
K^{22}_{\Theta,J}(p,q)=\sum_{k=0}^\infty p^k\left(-1-\Theta_{21}(p)\overline{\Theta_{21}(q)}
+\Theta_{22}(p)\overline{\Theta_{22}(q)}\right)\bar q^k
\label{5.5}
\end{equation}
and is also positive. Therefore,
\begin{align*}
K^{22}_{\Theta,J}(p,p)&=\sum_{k=0}^\infty
p^k\left(-1-|\Theta_{21}(p)|^2+|\Theta_{22}(p)|^2\right)\bar
p^k\\
&=\frac{-1-|\Theta_{21}(p)|^2+|\Theta_{22}(p)|^2}{1-|p|^2}\ge 0
\end{align*}
and in particular, $|\Theta_{22}(p)|>1$ for all $p\in\B$.
Therefore, its slice regular inverse $f=\Theta_{22}^{-\star}$ is
defined on $\B$ as well as the
function $g=f\star
\Theta_{21}=\Theta_{22}^{-\star}\star\Theta_{21}$. Using for now
this compact notation, observe that the kernel $f\star
K^{22}_{\Theta,J}\star_r\overline{f}$ is positive on $\B\times\B$
by Proposition \ref{P:2.1} (part (3)). According to \eqref{5.5},
this kernel equals
\begin{align*}
&f(p)\star \left(\sum_{k=0}^\infty p^k\left(-1-\Theta_{21}(p)\overline{\Theta_{21}(q)}
+\Theta_{22}(p)\overline{\Theta_{22}(q)}\right)\bar q^k\right)\star_r\overline{f(q)}\\
&=\sum_{k=0}^\infty p^kf(p)\left(-1-\Theta_{21}(p)\overline{\Theta_{21}(q)}
+\Theta_{22}(p)\overline{\Theta_{22}(q)}\right)\overline{f(q)}\bar q^k\\
&=\sum_{k=0}^\infty p^k\left(1-f(p)\overline{f(q)}-g(p)\overline{g(q)}\right)\bar q^k\succeq 0,
\end{align*}
and thus, both $f$ and $g$ are in $\mathcal R(\mathbb  B,\overline{\mathbb B})$.
\end{proof}
\begin{theorem}
Let us assume that $P>0$ and let
$\Theta=\left[\begin{smallmatrix}\Theta_{11} & \Theta_{12}\\
\Theta_{21} & \Theta_{22}\end{smallmatrix}\right]$ be defined as in \eqref{5.1}.
Then all solutions $S$ to the problem $\PR$ are given by the formula
\begin{equation}
S=(\Theta_{11}\star\cE+\Theta_{12})\star (\Theta_{21}\star\cE+\Theta_{22})^{-\star}
\label{5.6}
\end{equation}
with the free parameter $\cE$ running through the class $\mathcal R(\mathbb  B,\overline{\mathbb B})$.
\label{T:2.6}
\end{theorem}
\begin{proof}
By Proposition \ref{P:2.5}, the function $\Theta_{22}$ is left
$\star$--invertible and $\Theta_{22}^{-\star}\star
\Theta_{21}\in\mathcal R(\mathbb  B,\overline{\mathbb B})$. It is seen from formula \eqref{5.1}
that $\Theta$ is continuous on the closed unit ball
$\overline{\mathbb B}$ and that $\Theta(1)={\bf I}_2$. Therefore
$\Theta_{21}(1)=0$, $\Theta_{22}(1)=1$ and therefore
$\Theta_{22}^{-\star}\star \Theta_{21}$ is not a unimodular
constant. Hence, $|\Theta_{22}^{-\star}\star \Theta_{21}(p)|<1$
by the maximum modulus principle. Therefore,
$|\Theta_{22}^{-\star}\star \Theta_{21}\star \cE(p)|< 1$ for all
$p\in \B$ and for any  $\cE\in\mathcal R(\mathbb  B,\overline{\mathbb B})$. Consequently, the
function
$$
\Theta_{21}\star\cE+\Theta_{22}=\Theta_{22}\star\left(\Theta_{22}^{-\star}\star\Theta_{21}\star\cE+1\right)
$$
is $\star$--invertible and the formula \eqref{5.6} makes sense for every $\cE\in\mathcal R(\mathbb
B,\overline{\mathbb B})$.

\smallskip

By Theorem \ref{T:2.3}, a function $S: \, \mathbb B\to \mathbb H$ solves
the problem $\PR$ if and only if the kernel \eqref{4.1} is positive, which in turn
is equivalent (by part (3) in Proposition \ref{P:2.1}) to
\begin{equation}
\widetilde{K}_S(p,q):=K_S(p,q)-B^S(p)P^{-1}B^S(q)^*\succeq 0 \quad (p,q\in\mathbb B).
\label{5.7}
\end{equation}
Multiplying both parts in \eqref{5.4} by $\begin{bmatrix} 1 & -S\end{bmatrix}$  on the
left and by its adjoint on the right and taking into account \eqref{4.10} we get
$$
\begin{bmatrix} 1 & -S(p)\end{bmatrix}\star  K_{\Theta,J}(p,q)\star_r\begin{bmatrix}1  \\
-\overline{S(q)}\end{bmatrix}=B^S(p)P^{-1}B^S(q)^*.
$$
On the other hand, the kernel $K_S$ in \eqref{3.9} can be written as
$$
K_S(p,q)=\begin{bmatrix}1 & -S(p)\end{bmatrix}\star \left(\sum_{k=0}^\infty p^kJ\bar q^k\right)
\star_r
\begin{bmatrix}1  \\ -\overline{S(q)}\end{bmatrix}.
$$
Substituting the two latter representations into the right side of \eqref{5.7}
and taking into account the formula \eqref{5.4} for $K_{\Theta,J}$  gives
\begin{align}
\widetilde{K}_S(p,q)&=
\begin{bmatrix}1 & -S(p)\end{bmatrix}\star \left(
\sum_{k=0}^\infty p^k \Theta(p)J\Theta(q)^*\bar q^k\right)\star_r
\begin{bmatrix}1  \\ -\overline{S(q)}\end{bmatrix}\notag \\
&=\sum_{k=0}^\infty p^k \begin{bmatrix}1 & -S(p)\end{bmatrix}\star
\Theta(p)J\Theta(q)^*\star_r\begin{bmatrix}1 \\   -\overline{S(q)}\end{bmatrix} \bar  q^k\succeq
0.\label{5.8}
\end{align}
It remains to show that $S$ satisfies inequality \eqref{5.8} if and only if it is of the form
\eqref{5.4}
for some $\cE\in\mathcal R(\mathbb  B,\overline{\mathbb B})$. For the "only if" direction, let us assume
that \eqref{5.8} holds and  let
us introduce the functions
\begin{equation}
u=\Theta_{11}-S\star\Theta_{21}\quad\quad\mbox{and} \quad v=\Theta_{12}-S\star\Theta_{22}
\label{5.9}
\end{equation}
so that  $\begin{bmatrix} u & v\end{bmatrix}=\begin{bmatrix} 1 & -S\end{bmatrix}\star \Theta$.
Substituting the latter equality  into \eqref{5.8} and making use of the formula for $J$ in
\eqref{5.2}
we get
\begin{equation}
\widetilde{K}_{S}(p,q):=\sum_{k=0}^\infty p^k(u(p)\overline{u(q)}-v(p)\overline{v(q)})\bar
q^k\succeq 0 \quad
(p,q\in\mathbb B).
\label{5.10}
\end{equation}
Since $\Theta_{11}(1)=1$, $\Theta_{21}(1)=0$ (by formula \eqref{5.1}) and since $|S(p)|\le 1$ for
all
$p\in\mathbb B$, it follows that
$$
\limsup_{r\rightarrow 1^-}|u(r)|\ge \limsup_{r\rightarrow 1^-}
\left(|\Theta_{11}(r)|-|S(r)|\cdot |\Theta_{21}(r)|\right)=1.
$$
By continuity, $u$ is not vanishing in a real interval $[r_1, \, r_2]$ near $1$ and therefore, by
compactness, on an
open set $V\subset \mathbb B$ containing this interval. Therefore we may introduce the function
$\cE:=u^{-\star}\star v$
and rewrite \eqref{5.10} (at least for $p,q\in V$) in terms of this function as
$$
u(p)\star\left(\sum_{k=0}^\infty p^k\left(1-\cE(p)\overline{\cE(q)}\right)\bar
q^k\right)\star_r\overline{u(q)}\succeq 0
\quad (p,q\in V).
$$
The inverses $u^{-\star}$ and $u^{-\star_r}$ exist on $V$ and we conclude by part (3) in Proposition
\ref{P:2.1} that
$$
\sum_{k=0}^\infty p^k(1-\cE(p)\overline{\cE(q)})\bar q^k\succeq 0 \quad (p,q\in V).
$$
By Remark \ref{R:3.2} $\cE$ can be extended to a function from $\mathcal R(\mathbb B,\overline{\mathbb B})$,
which we still denote by $\cE$. By
the uniqueness
theorem, the equality
\begin{equation}
v=u\star\cE
\label{5.11}
\end{equation}
holds on the whole $\mathbb B$. Substituting equalities \eqref{5.9} into \eqref{5.11} gives
$$
\Theta_{12}-S\star\Theta_{22}=(\Theta_{11}-S\star\Theta_{21})\star\cE
$$
which can be written as
\begin{equation}
S\star(\Theta_{21}\star\cE +\Theta_{22})=\Theta_{11}\star\cE+\Theta_{12}.
\label{5.12}
\end{equation}
Since the function $\Theta_{21}\star\cE +\Theta_{22}$ is slice invertible for any $\cE\in\mathcal
S$,
the latter equality implies \eqref{5.6}.

\smallskip

Conversely, if $S$ is of the form \eqref{5.6} for some parameter $\cE\in \mathcal R(\mathbb
B,\overline{\mathbb B})$, then
equivalently, $S$ and $\cE$ are related as in \eqref{5.12}. This means that $u$ and $v$ defined as
in
\eqref{5.9} satisfy equality \eqref{5.11}. Then the formula \eqref{5.10} for $\widetilde{K}_S$ takes
the form
$$
\widetilde{K}_{S}(p,q)=\sum_{k=0}^\infty
p^ku(p)\left(1-\cE(p)\cE(q)^*\right)\overline{u(q)}\bar q^k=u(p)\star K_\cE(p,q)\star_r
\overline{u(q)}
$$
and is positive by Proposition \ref{P:2.1} (part (3)), since $\cE$ belongs to $\mathcal R(\mathbb
B,\overline{\mathbb B})$ so that  $K_\cE(p,q)\succeq 0$.
Thus, inequality \eqref{5.8} holds which completes the proof.
\end{proof}
\subsection{Schwarz-Pick inequalities} The goal of this subsection is to demonstrate that even
the single-point version of  Theorem \ref{2.6} provides some non-trivial information.
Let us observe that in case $n=1$, the formulas
\eqref{1.6} and \eqref{4.9} amount to
$P=\frac{1-|s_1|^2}{1-|p_1|^2}$,  $T=p_1$,  $E=1$, $N=s_1$ so that the formula
\eqref{5.1} simplifies to
\begin{equation}
\Theta(p)={\bf I}_2+(p-1)\sum_{k=0}^\infty p^k \begin{bmatrix}1 \\
\overline{s}_1\end{bmatrix}\overline{p}_1^k \frac{1-|p_1|^2}{1-|s_1|^2}(1-p_1)^{-1}
\begin{bmatrix} 1 & -s_1\end{bmatrix}.
\label{5.13}
\end{equation}
Upon specifying Theorem \ref{2.6} to the single-point case and
we conclude: {\em Given $p_1, \, s_1\in \mathbb B$, all functions
$S\in\mathcal R(\mathbb B,\mathbb B)$ mapping $p_1$ to $s_1$ are given by the formula
\eqref{5.6} where $\cE$ is the parameter from $\mathcal R(\mathbb B,\overline{\mathbb B})$ and
where $\Theta$ is given as in \eqref{5.13}.}

\smallskip

We next observe that for any $S$ of the form \eqref{5.6} with $\Theta$ given by \eqref{5.13},
\begin{align}
(S(p)-s_1)\star(1-\overline{s}_1\star
S(p))^{-\star}=&(p-p_1)\star(1-p\overline{p}_1)^{-\star}\gamma\notag\\
&\quad\star
(\cE(p)-s_1)\star (1-\overline{s}_1\star\cE(p))^{-\star}\label{5.14}
\end{align}
where we have set for short $\; \gamma=(1-\overline{p}_1)(1-p_1)^{-1}$.
Indeed, upon substituting  the linear fractional formula \eqref{5.6} for $S$ into the left
hand side of \eqref{5.14} and canceling out the factors
$(\Theta_{21}\star\cE+\Theta_{22})^{-\star}$ we get
\begin{align}
(S-s_1)\star(1-\overline{s}_1\star
S)^{-\star}=&(\Theta_{11}\star\cE+\Theta_{12}-s_1\star(\Theta_{21}\star\cE+\Theta_{22}))
\notag\\
&\quad\star (\Theta_{21}\star\cE+\Theta_{22}-\overline{s}_1
\star(\Theta_{11}\star\cE+\Theta_{12}))^{-\star}\notag\\
=&\begin{bmatrix}1 & -s_1 \end{bmatrix}\star \Theta\star \begin{bmatrix}\cE \\ 1
\end{bmatrix}
\star\left(\begin{bmatrix}-\overline{s}_1 & 1\end{bmatrix}\star \Theta\star
\begin{bmatrix}\cE \\ 1\end{bmatrix}\right)^{-\star}.\label{5.15}
\end{align}
Furthermore, it follows from \eqref{5.13} by direct verifications that
\begin{align*}
\begin{bmatrix}1 & -s_1 \end{bmatrix}\star \Theta\star \begin{bmatrix}\cE \\
1\end{bmatrix}(p)
&=(p-p_1)\star(1-p\overline{p}_1)^{-\star}\gamma\star (\cE(p)-s_1), \\
\begin{bmatrix}-\overline{s}_1 & 1\end{bmatrix}\star \Theta\star \begin{bmatrix}\cE \\
1\end{bmatrix}(p)&=1-\overline{s}_1\star\cE(p),
\end{align*}
and substituting the two last equalities into \eqref{5.15} gives \eqref{5.14}.
The Schwarz-Pick lemma for slice regular functions established recently in \cite{bist}
is an immediate consequence of \eqref{5.14}.
\begin{lemma}
For any $S\in\mathcal R(\mathbb B,\mathbb B)$ and $p_1\in\mathbb B$,
\begin{equation}
|(S(p)-S(p_1))\star(1-\overline{S(p_1)}\star
S(p))^{-\star}|\le |(p-p_1)\star(1-p\overline{p}_1)^{-\star}|
\label{5.16}
\end{equation}
with equality holding if and only if $S$ is an automorphism of $\mathbb B$.
\label{L:5.3}
\end{lemma}
\begin{proof}
The function $S$ solves the interpolation problem with the single interpolation node $p_1$ and
the target value $s_1:=S(p_1)$. Therefore, $S$ is of the form \eqref{5.6} for some
$\cE\in \mathcal R(\mathbb B,\overline{\mathbb B})$, and identity \eqref{5.14} holds
with $S(p_1)$ instead of $s_1$. Since $\cE\in\mathcal R(\mathbb B,\overline{\mathbb B})$,
we have
$$
|(\cE(p)-s_1)\star (1-\overline{s}_1\star\cE(p))^{-\star}|\le 1
$$
with equality if and only if $|\cE(p)|=1$; see \cite{acs2}. Since
$|\gamma|=|(1-\overline{p}_1)(1-p)^{-1})|=1$, we conclude from \eqref{5.14} that
inequality \eqref{5.16} holds with equality if and only if (by the maximum modulus principle)
$\cE$ is a unimodular constant function. The latter is equivalent (as it is easily seen
again from \eqref{5.14}) to $S$ be an automorphism of the unit ball. \end{proof}
\begin{remark}
Letting $s_1=0$ in formula (\ref{5.14}) we get Schwarz lemma:
If $S\in\mathcal R(\mathbb B, \mathbb B)$ vanishes at $p_0\in\mathbb B$, then $S$ is equal to the Blaschke
factor multiplied by some function $\mathcal E\in \mathcal R(\mathbb B, \overline{\mathbb B})$.
\end{remark}

\section{The determinate case}
\setcounter{equation}{0}

Still assuming that none three of
interpolation nodes belong to the same $2$-sphere, we will assume in addition  that the Pick
matrix  $P$ of the problem has rank $d<n$. In fact, it can be
shown  that any $d\times d$ principal submatrix of $P$ is
positive definite. Instead of proving this result which is beyond
the scope of this paper, we will permute indices (if necessary)
and assume without loss of generality that the {\em leading}
principal $d\times d$ submatrix of $P$ is invertible. In order to
keep notation from the previous section we proceed slightly
differently. We extend the problem $\PR$ to the problem
$\widetilde\PR$ by $k$ additional conditions
\begin{equation}
S(p_{n+i})=s_{n+i}\quad (i=1,\ldots,d)
\label{6.1}
\end{equation}
still assuming that no three interpolation nodes from the extended set $\{p_1,\ldots,p_{n+d}\}$
belong
to the  same
$2$-sphere, that the Pick matrix of the extended problem $\widetilde\PR$ (with interpolation
conditions
\eqref{1.5} {\em and} \eqref{6.1}) is positive semidefinite
\begin{equation}
{\mathbb P}=\begin{bmatrix}P & P_1^* \\ P_1 & P_2\end{bmatrix}
=\left[\sum_{k=0}^\infty p_i^k(1-s_i\bar s_j)\bar p_j^k\right]_{i,j=1}^{n+d}\ge 0,
\label{6.2}
\end{equation}
and that
\begin{equation}
{\rm rank} \, {\mathbb P}={\rm rank} \, P=n.
\label{6.3}
\end{equation}
\begin{theorem}
Under assumptions \eqref{6.2} and \eqref{6.3}, the problem  $\widetilde\PR$ has a unique solution.
\label{T:6.1}
\end{theorem}
\begin{proof}
Let $\Theta$ be defined as in \eqref{5.1}; the formula makes sense
since $P$ is invertible. Any solution $S$ to the extended problem $\widetilde\PR$
(if exists)  is also a solution to the problem $\PR$, so that it is necessarily of the form
\eqref{5.6} for some parameter $\cE\in\mathcal R(\mathbb B,\overline{\mathbb B})$.
The functions $S$ and $\cE$ are related as in \eqref{5.12} or (which is the same) as in
\eqref{5.11},
where $u$ and $v$ are defined as in \eqref{5.9}. Evaluating \eqref{5.11} at $p=p_{n+i}$ implies that
$S$ of the form \eqref{5.6} satisfies the additional interpolation conditions \eqref{6.1}
if and only if the corresponding parameter $\cE$ satisfies conditions
\begin{equation}
v(p_{n+i})=u\star \cE(p_{n+i})=u(p_{n+i})\cE(u(p_{n+i})^{-1}p_{n+i}u(p_{n+i}))
\label{6.4}
\end{equation}
for $i=1,\ldots,d$, where according to \eqref{5.9} and \eqref{6.1},
\begin{align}
u_{n+i}:=u(p_{n+i})&=\Theta_{11}(p_{n+i})-s_{n+i}\Theta_{21}(s_{n+i}^{-1}p_{n+i}s_{n+i}),\label{6.5}\\
v_{n+i}:=v(p_{n+i})&=\Theta_{12}(p_{n+i})-s_{n+i}\Theta_{22}(s_{n+i}^{-1}p_{n+i}s_{n+i}).\label{6.6}
\end{align}
Let us assume for a moment that the numbers defined in \eqref{6.5}, \eqref{6.6} are subject to
relations
\begin{equation}
|u_{n+i}|=|u_{n+i}|\neq 0,\quad u_{n+i}^{-1}v_{n+i}=u_{n+j}^{-1}v_{n+j}=\gamma\in\partial \mathbb B
\label{6.7}
\end{equation}
for all $i,j=1,\ldots,k$. We then conclude from \eqref{6.4} that in order for $S$ to be a solution
to
the
extended problem $\widetilde\PR$, it is necessary and sufficient that $S$ is of
the
form \eqref{5.6} for some $\cE\in\mathcal R(\mathbb  B,\overline{\mathbb B})$ such that
$$
\cE(u_{n+i}^{-1}p_{n+i}u_{n+i})=\gamma\quad\mbox{for}\quad i=1,\ldots,d.
$$
Since $|\gamma|=1$, it then follows by the maximum modulus principle that
a unique $\cE\in\mathcal R(\mathbb  B,\overline{\mathbb B})$ satisfying the latter conditions is the
constant function $\cE\equiv
\gamma$.

\smallskip

We now verify \eqref{6.7}. Let, in analogy to \eqref{4.9},
\begin{equation}
\widetilde T=\left[\begin{array}{ccc}p_{n+1} &&0
\\ &\ddots & \\ 0 && p_{n+d}\end{array}\right],\quad
\widetilde E=\left[\begin{array}{c}1\\ \vdots \\ 1 \end{array}\right],\quad
\widetilde N=\left[\begin{array}{c}s_{n+1} \\ \vdots\\ s_{n+d}\end{array}\right]
\label{6.8}
\end{equation}
so that the block entries $P_1$ and $P_2$ in \eqref{6.2} can be alternatively defined as unique
solutions to
the Stein equations
\begin{equation}
P_1-\widetilde T P_1 T^*=\widetilde EE^*-\widetilde NN^*, \qquad P_2-\widetilde T P_2 \widetilde
T^*=
\widetilde E\widetilde E^*-\widetilde N\widetilde N^*.
\label{6.9}
\end{equation}
Equating the $i$-th rows in the first of the two last equalities we conclude that
the $i$-th row $P_{1i}$ of $P_1$ satisfies
\begin{equation}
P_{1i}-p_{n+i}P_{1i}T^*=E^*-s_{n+i}N^*
\label{6.10}
\end{equation}
and is recovered from this equality by the formula
\begin{equation}
P_{1i}=\sum_{k=0}^\infty p_{n+i}^k(E^*-s_{n+i}N^*)T^{*k}
\label{6.11}
\end{equation}
Similarly, one can see from the second equality in \eqref{6.9} that the $ij$-entry of $P_2$
satisfies
the linear equation
\begin{equation}
P_{2,ij}-p_{n+i}P_{2,ij}\overline{p}_{n+j}=1-s_{n+i}\overline{s}_{n+j}.
\label{6.12}
\end{equation}
We now plug in formula \eqref{5.1} into \eqref{6.5} and \eqref{6.6} and then make use of
\eqref{6.11} to get
\begin{align}
u_{n+i}&=1+(p_{n+i}-1)\cdot \sum_{k=0}^\infty
p_{n+i}^k\left(E^*-s_{n+i}N^*\right)T^{*k}P^{-1}(I-T)^{-1}E\notag \\
&=1+(p_{n+i}-1)\cdot P_{1i}P^{-1}(I-T)^{-1}E,\label{6.13}\\
v_{n+i}&=-s_{n+i}-(p_{n+i}-1)\cdot \sum_{k=0}^\infty
p_{n+i}^k\left(E^*-s_{n+i}N^*\right)T^{*k}P^{-1}(I-T)^{-1}N\notag\\
&=-s_{n+i}-(p_{n+i}-1)\cdot P_{1i}P^{-1}(I-T)^{-1}N,\label{6.14}
\end{align}
so that
\begin{align*}
\begin{bmatrix}u_{n+i} & v_{n+i}\end{bmatrix}&=
\begin{bmatrix}1 & -s_{n+i}\end{bmatrix}+(p_{n+i}-1)\cdot P_{1i}P^{-1}(I-T)^{-1}
\begin{bmatrix}E & N\end{bmatrix}
\end{align*}
%where we have set $$R_i=E^*-s_{n+i}N^* \quad (i=1,\ldots,n)$$  for short.
We next use the two latter formulas and the formula \eqref{5.2} for $J$ to compute
\begin{align*}
&u_{n+i}\overline{u}_{n+j}-v_{n+i}\overline{v}_{n+j}=\begin{bmatrix}u_{n+i} &
v_{n+i}\end{bmatrix}J\begin{bmatrix}\overline{u}_{n+j} \\
\overline{v}_{n+j}\end{bmatrix}\notag\\
&=1-s_{n+i}\overline{s}_{n+j}+(p_{n+i}-1)\cdot P_{1i}
P^{-1}(I-T)^{-1}(E-N\overline{s}_{n+j}) \notag \\
&\qquad+(E^*-s_{n+i}N^*)(I-T^*)^{-1}P^{-1}P_{1,j}^*(\overline{p}_{n+j}-1)\notag \\
&\qquad+(p_{n+i}-1)\cdot P_{1i}P^{-1}(I-T)^{-1}\left (EE^*-NN^*\right)\notag\\
&\qquad\quad \times (I-T^*)^{-1}P^{-1}P_{1j}^*
(\overline{p}_{n+j}-1).\notag
\end{align*}
The latter expression can be further simplified due to \eqref{6.10} and the
equality
$$
P^{-1}(I-T)^{-1}\left (EE^*-NN^*\right)(I-T^*)^{-1}P^{-1}=
(I-T^*)^{-1}(P^{-1}-T^*P^{-1}T)(I-T)^{-1}
$$
which is a fairly straightforward consequence of the Stein identity \eqref{4.8}, as follows:
\begin{align}
&u_{n+i}\overline{u}_{n+j}-v_{n+i}\overline{v}_{n+j}\notag\\
&=1-s_{n+i}\overline{s}_{n+j}+(p_{n+i}-1)\cdot P_{1i}
P^{-1}(I-T)^{-1}(P_{ij}^*-TP_{1j}^*\overline{p}_{n+j}) \notag \\
&\qquad+(P_{1i}-p_{n+i}P_{1i}T^*)(I-T^*)^{-1}P^{-1}P_{1,j}^*(\overline{p}_{n+j}-1)\notag \\
&\qquad +(p_{n+i}-1)\cdot P_{1i}
(I-T^*)^{-1}(P^{-1}-T^*P^{-1}T)(I-T)^{-1}P_{1j}^*(\overline{p}_{n+j}-1).\notag
\end{align}
Further simplification follows thanks to \eqref{6.12} and the equality
$$
(P_{1i}-p_{n+i}P_{1i}T^*)(I-T^*)^{-1}=p_{n+i}P_{1i}-(p_{n+i}-1)P_{1i}(I-T^*)^{-1}.
$$
We get
\begin{align}
u_{n+i}\overline{u}_{n+j}-v_{n+i}\overline{v}_{n+j}=
&1-s_{n+i}\overline{s}_{n+j}+p_{n+i}P_{1i}P^{-1}P_{1j}^*(\overline{p}_{n+j}-1)\notag\\
&+(p_{n+i}-1)P_{1i}P^{-1}P_{1j}^*\overline{p}_{n+j}\notag \\
&-(p_{n+i}-1)P_{1i}P^{-1}P_{1j}^*(\overline{p}_{n+j}-1)\notag\\
=& P_{2,ij}-p_{n+i}P_{2,ij}\overline{p}_{n+j}+p_{n+i}P_{1i}P^{-1}P_{1j}^*\overline{p}_{n+j}
-P_{1i}P^{-1}P_{1j}^*\notag\\
=&P_{2,ij} -P_{1i}P^{-1}P_{1j}^*-p_{n+i}(P_{2,ij} -P_{1i}P^{-1}P_{1j}^*)\overline{p}_{n+j}.
\label{6.15}
\end{align}
Due to factorization
$$
{\mathbb P}=\begin{bmatrix}I_n &0 \\ P_1 P^{-1}& I_d\end{bmatrix}
\begin{bmatrix}P & 0 \\ 0 & P_2-P_1 P^{-1}P_1^*\end{bmatrix}
\begin{bmatrix}I_n & P^{-1}P_1 \\ 0 & I_d\end{bmatrix},
$$
the rank condition \eqref{6.3} implies $P_2=P_1 P^{-1}P_1^*$ or entry-wise,
$$
P_{2,ij}=P_{1i}P^{-1}P_{1j}\quad\mbox{for}\quad i,j=1,\ldots,d
$$
which together with \eqref{6.15} implies
\begin{equation}
u_{n+i}\overline{u}_{n+j}=v_{n+i}\overline{v}_{n+j}\quad\mbox{for}\quad i,j=1,\ldots,d.
\label{6.16}
\end{equation}
Letting $i=j$ in the latter equalities gives $|u_{n+i}|=|v_{n+i}|$ for $i=1,\ldots,d$.
To show that $u_{n+i}$ and $v_{n+i}$ are nonzero we will argue via contradiction.
Assuming that
\begin{equation}
u_{n+i}=v_{n+i}=0.
\label{6.17}
\end{equation}
for some $i\in\{1,\ldots,d\}$ we then get
\begin{align}
0=&u_{n+i}E^*+v_{n+i}N^*\notag\\
=&E^*-s_{n+i}N^*+(p_{n+i}-1)\cdot P_{1i}P^{-1}(I-T)^{-1}(EE^*-NN^*)\notag\\
=&P_{1i}-p_{n+i}P_{1i}T^*+(p_{n+i}-1)\cdot P_{1i}P^{-1}(I-T)^{-1}(P-TPT^*)\notag\\
=&\left(p_{n+i}P_{1i}P^{-1}-P_{1i}P^{-1}T\right)(I-T)^{-1}P(I-T^*).\label{6.18}
\end{align}
We remark that the second equality in the latter computation follows from
formulas \eqref{6.13}, \eqref{6.14},  the third equality is a consequence of
relations \eqref{4.8} and \eqref{6.10} while the last equality is easily verified directly.
Since the matrices $P$ and $I-T$ are invertible (recall that $P$ is Hermitian and $T$ is diagonal),
it follows from \eqref{6.18} that
\begin{equation}
p_{n+i}P_{1i}P^{-1}=P_{1i}P^{-1}T.
\label{6.19}
\end{equation}
Substituting the latter equality into \eqref{6.13}, \eqref{6.14} results in
$$
u_{n+i}1=1+P_{1i}P^{-1}T(I-T)^{-1}E- P_{1i}P^{-1}(I-T)^{-1}E=1-P_{1i}P^{-1}E
$$
which being combined with the assumption \eqref{6.17}, leads us to
\begin{equation}
P_{1i}P^{-1}E=1.
\label{6.20}
\end{equation}
Let ${\bf e}_j$ denote the $i$-th column in the identity matrix ${\bf I}_n$.
Multiplying both sides in
\eqref{6.18} by ${\bf e}_j$ on the right and making use of the diagonal structure \eqref{4.9}
of $T$ we get
\begin{equation}
p_{n+i}P_{1i}P^{-1}{\bf e}_j=P_{1i}P^{-1}{\bf e}_jp_j\quad\mbox{for}\quad j=1,\ldots,n.
\label{6.21}
\end{equation}
Therefore,
\begin{equation}
P_{1i}P^{-1}{\bf e}_j=0,\quad\mbox{whenever}\quad p_{j}\not\in [p_{n+i}].
\label{6.22}
\end{equation}
Due to the assumption that  no three points from the set $\{p_1,\ldots,p_{n+d}\}$
belong to the  same $2$-sphere, the intersection of the set $\{p_1,\ldots,p_n\}$
with the $2$-sphere $[p_{n+i}]$ is either empty or a singleton. We will show that either case
leads to a contradiction.

\smallskip

{\bf Case 1.} If $p_j\not\in[p_{n+i}]$  for all $j=1,\ldots,n$, then it follows from \eqref{6.22}
that $P_{1i}P^{-1}=0$ which contradicts to \eqref{6.20}.

\smallskip

{\bf Case 2.} Without loss of generality we assume that $p_1\in[p_{n+i}]$ and
$p_j\not\in[p_{n+i}]$ for $j=2,\ldots,n$. Therefore, equalities \eqref{6.22} hold for
all $j=2,\ldots,n$ and then we conclude from \eqref{6.20}
$$
1=P_{1i}P^{-1}E=P_{1i}P^{-1}({\bf e}_1+{\bf e}_2+\ldots+{\bf e}_n)=P_{1i}P^{-1}{\bf e}_1.
$$
Due to this latter relation, the equality \eqref{6.21} for $j=1$ simplifies to
$p_{n+i}=p_1$ which contradicts to the assumption that all interpolation nodes are distinct.

\smallskip

The derived contradictions show that equalities \eqref{6.17} cannot be in force which
completes the proof of the first part in \eqref{6.7}. Once we know that $u_{n+1}\neq 0$,
the second part in \eqref{6.7} follows from  \eqref{6.16}.
\end{proof}

\bigskip

\bibliographystyle{amsplain}

\end{document}